\documentclass[10pt,reqno]{amsart}
\usepackage[hypertex]{hyperref}

\usepackage{amsmath,amsthm}
\usepackage{amssymb}
\usepackage{euscript}
\usepackage[frame,ps,matrix,arrow,curve,rotate,all,2cell,tips]{xy}

\numberwithin{equation}{subsection}

\newcommand{\sqsp}{\renewcommand{\baselinestretch}{1.1}\tiny\normalsize}

\newtheorem{theorem}[subsection]{Theorem}
\newtheorem{lemma}[subsection]{Lemma}

\newtheorem{corollary}[subsection]{Corollary}

\theoremstyle{definition}
\newtheorem{definition}[subsection]{Definition}
\newtheorem{example}[subsection]{Example}

\newcommand{\cata}{\mathcal{A}}

\newcommand{\bk}{\mathbf{k}}

\newcommand{\bZ}{\mathbf{Z}}

\newcommand{\rhotilde}{\widetilde{\rho}}

\begin{document}

\title{Module Hom-algebras}
\author{Donald Yau}

\begin{abstract}
We study a twisted version of module algebras called module Hom-algebras.  It is shown that module algebras deform into module Hom-algebras via endomorphisms.  As an example, we construct certain $q$-deformations of the usual $sl(2)$-action on the affine plane.
\end{abstract}

\keywords{Module Hom-algebra, module algebra.}

\subjclass[2000]{16W30, 16S30}

\address{Department of Mathematics\\
    The Ohio State University at Newark\\
    1179 University Drive\\
    Newark, OH 43055, USA}
\email{dyau@math.ohio-state.edu}

\date{\today}
\maketitle

\sqsp

\section{Introduction}


Let $H$ be a bialgebra, and let $A$ be an algebra.  An $H$-\emph{module algebra} structure on $A$ consists of an $H$-module structure such that the multiplication map on $A$ becomes an $H$-module morphism.  In other words, one has the \emph{module algebra axiom}
\begin{equation}
\label{eq:modulealg}
x(ab) = \sum_{(x)} \,(x'a)(x''b)
\end{equation}
for $x \in H$ and $a,b \in A$, where $\Delta(x) = \sum_{(x)} x' \otimes x''$ is the Sweedler's notation \cite{sweedler} for comultiplication.  Module algebras arise often in algebraic topology, quantum groups \cite[Chapter V.6]{kassel}, Lie and Hopf algebras theory \cite{mont,sweedler}, and group representations \cite[Chapter 3]{abe}.  For example, the singular mod $p$ cohomology $\mathrm{H}^*(X; \bZ/p)$ of a topological space $X$ is an $\cata_p$-module algebra, where $\cata_p$ is the Steenrod algebra associated to the prime $p$ \cite{es}.  Likewise, the complex cobordism $\mathrm{MU}^*(X)$ of a topological space $X$ is an $S$-module algebra, where $S$ is the Landweber-Novikov algebra \cite{landweber,novikov} of stable cobordism operations.

The purpose of this paper is to study a Hom-algebra analogue of module algebras, in which \eqref{eq:modulealg} is twisted by a linear map.  Hom-algebras were first defined for Lie algebras.  A \emph{Hom-Lie algebra} is a vector space $L$ together with a bilinear skew-symmetric bracket $[-,-] \colon L \otimes L \to L$ and a linear map $\alpha \colon L \to L$ such that $\alpha[x,y] = [\alpha(x),\alpha(y)]$ for $x,y \in L$ (multiplicativity) and that the following \emph{Hom-Jacobi identity} holds:
\[
[[x,y],\alpha(z)] + [[z,x],\alpha(y)] + [[y,z],\alpha(x)]] = 0.
\]
Hom-Lie algebras were introduced in \cite{hls} (without multiplicativity) to describe the structures on certain $q$-deformations of the Witt and the Virasoro algebras.  Earlier precursors of Hom-Lie algebras can be found in \cite{hu,liu}.  A Lie algebra is a Hom-Lie algebra with $\alpha = Id$.  More generally, if $L$ is a Lie algebra and $\alpha$ is a Lie algebra endomorphism on $L$, then $L$ becomes a Hom-Lie algebra with the bracket $[x,y]_\alpha = \alpha([x,y])$ \cite{yau2}.

Likewise, a \emph{Hom-associative algebra} $A$ \cite{ms,ms3} has a bilinear map $\mu \colon A \otimes A \to A$ and a linear map $\alpha \colon A \to A$ such that $\alpha(\mu(x,y)) = \mu(\alpha(x),\alpha(y))$ and
\begin{equation}
\label{eq:homass}
\mu(\alpha(x),\mu(y,z)) = \mu(\mu(x,y),\alpha(z))
\end{equation}
for $x,y,z \in A$.  It is shown in \cite{ms} that a Hom-associative algebra $(A,\mu,\alpha)$ gives rise to a Hom-Lie algebra $(A,[-,-],\alpha)$ in which $[x,y] = \mu(x,y) - \mu(y,x)$, i.e., the commutator bracket of $\mu$.  Conversely, given a Hom-Lie algebra $L$, there is a universal enveloping Hom-associative algebra $U(L)$ \cite{yau,yau3}.  Dualizing \eqref{eq:homass}, one can define Hom-coassociative coalgebra and Hom-bialgebra \cite{ms2,ms4,yau3}.  It is shown in \cite{yau3} that the universal enveloping Hom-associative algebra $U(L)$ is a Hom-bialgebra.  Given an algebra $A$ and an algebra endomorphism $\alpha$, one obtains a Hom-associative algebra structure on $A$ with multiplication $\mu_\alpha = \alpha \circ \mu$ \cite{yau2}.  The same procedure can be applied to coalgebras, bialgebras, and other kinds of algebraic structures, as was done in \cite{atm,ms4,yau3}, to obtain Hom-coassociative coalgebras, Hom-bialgebras, and other Hom-algebra structures.


Using \eqref{eq:homass} as a model, one can define modules and their morphisms over a Hom-associative algebra.  The precise definitions will be given in Section ~\ref{sec:modalg}.  Let $H = (H,\mu_H,\Delta_H,\alpha_H)$ be a Hom-bialgebra and $A = (A,\mu_A,\alpha_A)$ be a Hom-associative algebra.  Then an \emph{$H$-module Hom-algebra} structure on $A$ consists of an $H$-module structure $\rho \colon H \otimes A \to A$ on $A$, $\rho(x \otimes a) = xa$, such that the \emph{module Hom-algebra axiom}
\begin{equation}
\label{eq:modulehomalg}
\alpha_H^2(x)(ab) = \sum_{(x)} \, (x'a)(x''b)
\end{equation}
holds.  Of course, if $\alpha_H = Id_H$ and $\alpha_A = Id_A$, then an $H$-module Hom-algebra is exactly an $H$-module algebra.

As in the case of module algebras, the module Hom-algebra axiom \eqref{eq:modulehomalg} can be interpreted as a certain multiplication map being an $H$-module morphism.  The following characterization of module Hom-algebras will be proved in Section ~\ref{sec:modalg}.  The symbol $\tau_{H,A}$ denotes the twist isomorphism $H \otimes A \to A \otimes H$ where $\tau_{H,A}(x \otimes a) = a \otimes x$.

\begin{theorem}
\label{thm:char}
Let $H = (H,\mu_H,\Delta_H,\alpha_H)$ be a Hom-bialgebra, $A = (A,\mu_A,\alpha_A)$ be a Hom-associative algebra, and $\rho \colon H \otimes A \to A$ be an $H$-module structure on $A$.  Then the following statements hold.
\begin{enumerate}
\item
The map
\begin{equation}
\label{eq:rhotilde}
\rhotilde = \rho \circ (\alpha_H^2 \otimes Id_A) \colon H \otimes A \to A
\end{equation}
gives $A$ another $H$-module structure.
\item
The map
\begin{equation}
\label{eq:rho2}
\rho^2 = \rho^{\otimes 2} \circ (Id_H \otimes \tau_{H,A} \otimes Id_A) \circ (\Delta_H \otimes Id_A^{\otimes 2}) \colon H \otimes A^{\otimes 2} \to A^{\otimes 2}
\end{equation}
gives $A^{\otimes 2}$ an $H$-module structure.
\item
The map $\rho$ gives $A$ the structure of an $H$-module Hom-algebra if and only if $\mu_A \colon A^{\otimes 2} \to A$ is a morphism of $H$-modules, where in $\mu_A$ we equip $A^{\otimes 2}$ and $A$ with the $H$-module structures \eqref{eq:rho2} and \eqref{eq:rhotilde}, respectively.
\end{enumerate}
\end{theorem}

As we mentioned above, algebras, coalgebras, and bialgebras deform into the respective types of Hom-algebras via an endomorphism.  The following result, which will be proved in Section ~\ref{sec:deform}, shows that module algebras deform into module Hom-algebras via an endomorphism.  This result provides a large class of examples of module Hom-algebras.

\begin{theorem}
\label{thm:deform}
Let $H = (H,\mu_H,\Delta_H)$ be a bialgebra and $A = (A,\mu_A)$ be an $H$-module algebra via $\rho \colon H \otimes A \to A$.  Let $\alpha_H \colon H \to H$ be a bialgebra endomorphism and $\alpha_A \colon A \to A$ be an algebra endomorphism such that 
\begin{equation}
\label{eq:alpharho}
\alpha_A \circ \rho = \rho \circ (\alpha_H \otimes \alpha_A).  
\end{equation}
Write $H_\alpha$ for the Hom-bialgebra $(H,\mu_{\alpha,H} = \alpha_H \circ \mu_H,\Delta_{\alpha,H} = \Delta_H \circ \alpha_H,\alpha_H)$ and $A_\alpha$ for the Hom-associative algebra $(A,\mu_{\alpha,A} = \alpha_A \circ \mu_A,\alpha_A)$.  Then the map
\begin{equation}
\label{eq:rhoalpha}
\rho_\alpha = \alpha_A \circ \rho \colon H \otimes A \to A
\end{equation}
gives the Hom-associative algebra $A_\alpha$ the structure of an $H_\alpha$-module Hom-algebra.
\end{theorem}

We now describe some consequences of Theorem ~\ref{thm:deform}.  In the context of the above Theorem, if $\alpha_H = Id_H$, then we have the condition $\alpha_A \circ \rho = \rho \circ (Id_H \otimes \alpha_A)$, which means exactly that $\alpha_A$ is $H$-linear.  Thus, using the same notations as above, we have the following special case.

\begin{corollary}
\label{cor:deform}
Let $H = (H,\mu_H,\Delta_H)$ be a bialgebra, $A = (A,\mu_A)$ be an $H$-module algebra via $\rho \colon H \otimes A \to A$, and $\alpha_A \colon A \to A$ be an algebra endomorphism that is also $H$-linear.  Then the map $\rho_\alpha$ \eqref{eq:rhoalpha} gives the Hom-associative algebra $A_\alpha$ the structure of an $H$-module Hom-algebra, where $H$ denotes the Hom-bialgebra $(H,\mu_H,\Delta_H,Id_H)$.
\end{corollary}

Examples that illustrate Corollary ~\ref{cor:deform} will be given in Section ~\ref{sec:deform}.

Module algebras over the universal enveloping bialgebra $U(L)$ of a Lie algebra $L$ are important in the study of Lie algebras and quantum groups.  For example, there is a $U(sl(2))$-module algebra structure on the affine plane $\bk[x,y]$ that captures all the finite dimensional simple $sl(2)$-modules \cite[Theorem V.6.4]{kassel}.  The following result, which will be proved in Section ~\ref{sec:twisted}, applies Theorem ~\ref{thm:deform} in the enveloping bialgebra context.  The symbols $\mu_U \colon U(L)^{\otimes 2} \to U(L)$ and $\Delta_U \colon U(L) \to U(L)^{\otimes 2}$ denote the multiplication and the comultiplication, respectively.

\begin{theorem}
\label{thm:U}
Let $L$ be a Lie algebra, $A = (A,\mu_A)$ be a $U(L)$-module algebra via $\rho \colon U(L) \otimes A \to A$, $\alpha_L \colon L \to L$ be a Lie algebra endomorphism, and $\alpha_A \colon A \to A$ be an algebra endomorphism.  Suppose that $\alpha_A(xa) = \alpha_L(x)\alpha_A(a)$ for $x \in L$ and $a \in A$.  Then the following statements hold.
\begin{enumerate}
\item
There exists a unique bialgebra endomorphism extension $\alpha_U \colon U(L) \to U(L)$ of $\alpha_L$ such that
\begin{equation}
\label{eq:alphaza}
\alpha_A(za) = \alpha_U(z)\alpha_A(a)
\end{equation}
for $z \in U(L)$ and $a \in A$.
\item
The map
\[
\rho_\alpha = \alpha_A \circ \rho \colon U(L) \otimes A \to A
\]
gives the Hom-associative algebra $A_\alpha = (A,\mu_{\alpha,A} = \alpha_A \circ \mu_A,\alpha_A)$ the structure of a $U(L)_\alpha$-module Hom-algebra.  Here $U(L)_\alpha$ is the Hom-bialgebra $(U(L),\mu_{\alpha,U} = \alpha_U \circ \mu_U, \Delta_{\alpha,U} = \Delta_U \circ \alpha_U, \alpha_U)$.
\end{enumerate}
\end{theorem}

Using Theorem ~\ref{thm:U}, we will construct certain $q$-deformations of the $U(sl(2))$-module algebra structure on the affine plane mentioned above.

We note that module Hom-algebras, and other kinds of Hom-algebras (e.g., Hom-(co)associative (co)algebras, Hom-Lie (co)algebras, Hom-bialgebras  \cite{ms,ms2,ms3,ms4,yau,yau2,yau3}, and $n$-ary Hom-Nambu/Lie/associative algebras \cite{atm}), are algebras over their respective (colored) PROPs \cite{markl}.  In particular, their algebraic deformations, in the sense of Gerstenhaber, are governed by some $L_\infty$-deformation complexes, as shown in \cite{fmy}.

The rest of this paper is organized as follows.  In Section ~\ref{sec:modalg}, we state the relevant definitions and prove Theorem ~\ref{thm:char}.  In Section ~\ref{sec:deform}, we prove Theorem ~\ref{thm:deform} and provide some examples to illustrate Corollary ~\ref{cor:deform}.  In Section ~\ref{sec:twisted}, we prove Theorem ~\ref{thm:U} and illustrate it with the case $L = sl(2)$ and $A = \bk[x,y]$ (Example ~\ref{ex:q}).

\section{Hom-algebra analogue of module algebra}
\label{sec:modalg}


Before we define module Hom-algebras and prove Theorem ~\ref{thm:char}, let us first recall some basic definitions regarding Hom-modules, Hom-associative algebras, and Hom-bialgebras.  The first two parts of Theorem ~\ref{thm:char} will be proved as Lemmas ~\ref{lem:rhotilde} and ~\ref{lem:rho2}.  The last part will be proved at the end of this section.

\subsection{Conventions and notations}
\label{subsec:notations}

Throughout the rest of this paper, $\bk$ denotes a field of any characteristic.  The only exception is Example ~\ref{ex:q}, where $\bk$ is assumed to have characteristic $0$.  Vector spaces, tensor products, and linearity are all meant over $\bk$, unless otherwise specified.

Given two vector spaces $V$ and $W$, denote by $\tau = \tau_{V,W} \colon V \otimes W \to W \otimes V$ the twist isomorphism, i.e., $\tau(v \otimes w) = w \otimes v$.  
For a coalgebra $C$ with comultiplication $\Delta \colon C \to C \otimes C$, we use Sweedler's notation for comultiplication: $\Delta(x) = \sum_{(x)} x'\otimes x''$ \cite{sweedler}.

\subsection{Hom-modules}

A \textbf{Hom-module} is a pair $(V, \alpha)$ \cite{yau} in which $V$ is a vector space and $\alpha \colon V \to V$ is a linear map.  A morphism $(V, \alpha_V) \to (W, \alpha_W)$ of Hom-modules is a linear map $f \colon V \to W$ such that $\alpha_W \circ f = f \circ \alpha_V$.  We will often abbreviate a Hom-module $(V,\alpha)$ to $V$.  The tensor product of the Hom-modules $(V, \alpha_V)$ and $(W, \alpha_W)$ consists of the vector space $V \otimes W$ and the linear self-map $\alpha_V \otimes \alpha_W$.

\subsection{Hom-associative algebras}
\label{subsec:homass}

A \textbf{Hom-associative algebra} \cite{ms,ms3,yau2} is a triple $(A,\mu,\alpha)$ in which $(A,\alpha)$ is a Hom-module and $\mu \colon A \otimes A \to A$ is a bilinear map such that
\begin{enumerate}
\item
$\alpha \circ \mu = \mu \circ \alpha^{\otimes 2}$ (multiplicativity) and
\item
$\mu \circ (\alpha \otimes \mu) = \mu \circ (\mu \otimes \alpha)$ (Hom-associativity).
\end{enumerate}
In what follows, we will also write $\mu(a \otimes b)$ as $ab$.

For example, an algebra $(A,\mu)$ can be regarded as a Hom-associative algebra $(A,\mu,Id_A)$.  Conversely, let $\alpha \colon A \to A$ be an algebra endomorphism of the algebra $(A,\mu)$.  Define the new multiplication
\begin{equation}
\label{eq:mualpha}
\mu_\alpha = \alpha \circ \mu \colon A^{\otimes 2} \to A.
\end{equation}
One can check that $A_\alpha = (A,\mu_\alpha,\alpha)$ is a Hom-associative algebra \cite{yau2}.  In fact, both $\mu_\alpha \circ (\alpha \otimes \mu_\alpha)$ and $\mu_\alpha \circ (\mu_\alpha \otimes \alpha)$, when applied to $a \otimes b \otimes c \in A^{\otimes 3}$, are equal to $\alpha^2(abc)$.  So $\mu_\alpha$ is Hom-associative.  Multiplicativity of $\alpha$ with respect to $\mu_\alpha$ can be checked similarly.

Suppose that $(A,\mu_A,\alpha_A)$ and $(B,\mu_B,\alpha_B)$ are Hom-associative algebras.  Their tensor product $A \otimes B$ as a Hom-associative algebra is defined in the usual way, with $\alpha_{A \otimes B} = \alpha_A \otimes \alpha_B$ and
\[
\mu_{A \otimes B} = (\mu_A \otimes \mu_B) \circ (Id_A \otimes \tau_{B,A} \otimes Id_B).
\]
A \textbf{morphism} $f \colon (A,\mu_A,\alpha_A) \to (B,\mu_B,\alpha_B)$ of Hom-associative algebras is a morphism $f \colon (A,\alpha_A) \to (B,\alpha_B)$ of the underlying Hom-modules such that $f \circ \mu_A = \mu_B \circ f^{\otimes 2}$.

\subsection{Modules over Hom-associative algebras}
\label{subsec:modules}

Let $(A,\mu_A,\alpha_A)$ be a Hom-associative algebra and $(M,\alpha_M)$ be a Hom-module.  An \textbf{$A$-module} structure on $M$ consists of a morphism $\rho \colon A \otimes M \to M$ of Hom-modules, called the \textbf{structure map}, such that
\begin{equation}
\label{eq:moduleaxiom}
\rho \circ (\alpha_A \otimes \rho) = \rho \circ (\mu_A \otimes \alpha_M).
\end{equation}
We will also write $\rho(a \otimes m)$ as $am$ for $a \in A$ and $m \in M$.  In this notation, \eqref{eq:moduleaxiom} can be rewritten as
\begin{equation}
\label{eq:moduleaxiom'}
\alpha_A(a)(bm) = (ab)\alpha_M(m)
\end{equation}
for $a,b \in A$ and $m \in M$.  If $M$ and $N$ are $A$-modules, then a \textbf{morphism} of $A$-modules $f \colon M \to N$ is a morphism of the underlying Hom-modules such that 
\begin{equation}
\label{eq:modmorphism}
f \circ \rho_M = \rho_N \circ (Id_A \otimes f), 
\end{equation}
i.e., $f(am) = af(m)$.

The following Lemma will be needed when we give an alternative characterization of a module Hom-algebra.  It proves the first part of Theorem ~\ref{thm:char}.

\begin{lemma}
\label{lem:rhotilde}
Let $(A,\mu_A,\alpha_A)$ be a Hom-associative algebra and $(M,\alpha_M)$ be an $A$-module with structure map $\rho \colon A \otimes M \to M$.  Define the map
\begin{equation}
\label{eq:rhot}
\rhotilde = \rho \circ (\alpha_A^2 \otimes Id_M) \colon A \otimes M \to M.
\end{equation}
Then $\rhotilde$ is the structure map of another $A$-module structure on $M$.
\end{lemma}

\begin{proof}
The fact that $\rho$ is a morphism of Hom-modules means that
\begin{equation}
\label{eq:rhomorphism}
\alpha_M \circ \rho = \rho \circ (\alpha_A \otimes \alpha_M).
\end{equation}
To see that $\rhotilde$ is a morphism of Hom-modules, we compute as follows:
\[
\begin{split}
\alpha_M \circ \rhotilde
&= \alpha_M \circ \rho \circ (\alpha_A^2 \otimes Id_M)\\
&= \rho \circ (\alpha_A \otimes \alpha_M) \circ (\alpha_A^2 \otimes Id_M) \quad \text{by \eqref{eq:rhomorphism}}\\
&= \rho \circ (\alpha_A^2 \otimes Id_M) \circ (\alpha_A \otimes \alpha_M)\\
&= \rhotilde \circ (\alpha_A \otimes \alpha_M).
\end{split}
\]
To see that $\rhotilde$ satisfies \eqref{eq:moduleaxiom} (with $\rhotilde$ in place of $\rho$), we compute as follows:
\[
\begin{split}
\rhotilde \circ (\alpha_A \otimes \rhotilde)
&= \rho \circ (\alpha_A^2 \otimes Id_M) \circ (\alpha_A \otimes (\rho \circ (\alpha_A^2 \otimes Id_M)))\\
&= \rho \circ (\alpha_A \otimes \rho) \circ (\alpha_A^2 \otimes \alpha_A^2 \otimes Id_M)\\
&= \rho \circ (\mu_A \otimes \alpha_M) \circ (\alpha_A^2 \otimes \alpha_A^2 \otimes Id_M) \quad \text{by \eqref{eq:moduleaxiom}}\\
&= \rho \circ((\alpha_A^2 \circ \mu_A) \otimes \alpha_M) \quad \text{by multiplicativity of $\alpha_A$}\\
&= \rho \circ (\alpha_A^2 \otimes Id_M) \circ (\mu_A \otimes \alpha_M)\\
&= \rhotilde \circ (\mu_A \otimes \alpha_M).
\end{split}
\]
We have shown that $\rhotilde$ is the structure map of an $A$-module structure on $M$.
\end{proof}

\subsection{Hom-bialgebras}
\label{subsec:hombialgebra}

\begin{definition}
A \textbf{Hom-bialgebra} is a quadruple $(H,\mu,\Delta,\alpha)$ in which:
\begin{enumerate}
\item
$(H,\mu,\alpha)$ is a Hom-associative algebra.
\item
The comultiplication $\Delta \colon H \to H^{\otimes 2}$ is linear and is \emph{Hom-coassociative}, in the sense that
\begin{equation}
\label{eq:homcoass}
(\Delta \otimes \alpha) \circ \Delta = (\alpha \otimes \Delta) \circ \Delta.
\end{equation}
\item
$\Delta$ is a morphism of Hom-associative algebras.
\end{enumerate}
\end{definition}
Note that $\Delta$ being a morphism of Hom-associative algebras means that
\begin{equation}
\label{eq:Deltaalpha}
\Delta \circ \alpha = \alpha^{\otimes 2} \circ \Delta
\end{equation}
and
\begin{equation}
\label{eq:Deltamu}
\Delta \circ \mu = \mu^{\otimes 2} \circ (Id_H \otimes \tau \otimes Id_H) \circ \Delta^{\otimes 2}.
\end{equation}

The following Lemma will be needed when we give an alternative characterization of a module Hom-algebra.  It proves the second part of Theorem ~\ref{thm:char}.  By an $H$-module, we mean a module over the Hom-associative algebra $(H,\mu,\alpha)$.

\begin{lemma}
\label{lem:rho2}
Let $(H,\mu_H,\Delta_H,\alpha_H)$ be a Hom-bialgebra and $(M,\alpha_M)$ be an $H$-module with structure map $\rho \colon H \otimes M \to M$.  Define the map
\begin{equation}
\label{eq:rho2'}
\rho^2 = \rho^{\otimes 2} \circ (Id_H \otimes \tau_{H,M} \otimes Id_M) \circ (\Delta_H \otimes Id_M^{\otimes 2}) \colon H \otimes M^{\otimes 2} \to M^{\otimes 2}.
\end{equation}
Then $\rho^2$ is the structure map of an $H$-module structure on $M^{\otimes 2}$.
\end{lemma}

\begin{proof}
To see that $\rho^2$ is a morphism of Hom-modules, we compute as follows:
\[
\begin{split}
\alpha_M^{\otimes 2} \circ \rho^2 
&= \alpha_M^{\otimes 2} \circ \rho^{\otimes 2} \circ (Id_H \otimes \tau_{H,M} \otimes Id_M) \circ (\Delta_H \otimes Id_M^{\otimes 2})\\
&= \rho^{\otimes 2} \circ (\alpha_H \otimes \alpha_M)^{\otimes 2} \circ (Id_H \otimes \tau_{H,M} \otimes Id_M) \circ (\Delta_H \otimes Id_M^{\otimes 2}) \quad \text{by \eqref{eq:rhomorphism}}\\
&= \rho^{\otimes 2} \circ (Id_H \otimes \tau_{H,M} \otimes Id_M) \circ (\Delta_H \otimes Id_M^{\otimes 2}) \circ (\alpha_H \otimes \alpha_M^{\otimes 2}) \quad \text{by \eqref{eq:Deltaalpha}}\\
&= \rho^2 \circ (\alpha_H \otimes \alpha_M^{\otimes 2}).
\end{split}
\]
To see that $\rho^2$ satisfies \eqref{eq:moduleaxiom} (with $\rho^2$, $H$, and $M^{\otimes 2}$ in place of $\rho$, $A$, and $M$, resp.), we compute as follows, where some obvious subscripts have been left out:
\[
\begin{split}
&\rho^2 \circ (\alpha_H \otimes \rho^2)\\
&= \rho^{\otimes 2} \circ (Id_H \otimes \tau \otimes Id_M) \circ (\Delta \otimes Id_M^{\otimes 2}) \circ \left\{\alpha_H \otimes (\rho^{\otimes 2} \circ (Id_H \otimes \tau \otimes Id_M) \circ (\Delta \otimes Id_M^{\otimes 2}))\right\}\\
&= \rho^{\otimes 2} \circ (Id \otimes \tau \otimes Id) \circ (\alpha_H^{\otimes 2} \otimes Id_M^{\otimes 2}) \circ \left\{\Delta \otimes (\rho^{\otimes 2} \circ (Id_H \otimes \tau \otimes Id_M) \circ (\Delta \otimes Id_M^{\otimes 2}))\right\} ~ \text{by \eqref{eq:Deltaalpha}}\\
&= \rho^{\otimes 2} \circ (Id_H \otimes \tau \otimes Id_M) \circ (\mu^{\otimes 2} \otimes Id_M^{\otimes 2}) \circ (Id_H \otimes \tau \otimes Id_H \otimes Id_M^{\otimes 2}) \circ (\Delta^{\otimes 2} \otimes \alpha_M^{\otimes 2}) \quad \text{by \eqref{eq:moduleaxiom}}\\
&= \rho^{\otimes 2} \circ (Id_H \otimes \tau \otimes Id_M) \circ (\Delta \otimes Id_M^{\otimes 2}) \circ (\mu \otimes \alpha_M^{\otimes 2}) \quad \text{by \eqref{eq:Deltamu}}\\
&= \rho^2 \circ (\mu \otimes \alpha_M^{\otimes 2}).
\end{split}
\]
We have shown that $\rho^2$ is the structure map of an $H$-module structure on $M^{\otimes 2}$.
\end{proof}

\subsection{Module Hom-algebra}
\label{subsec:mha}

Let $(H,\mu_H,\Delta_H,\alpha_H)$ be a Hom-bialgebra and $(A,\mu_A,\alpha_A)$ be a Hom-associative algebra.  An \textbf{$H$-module Hom-algebra} structure on $A$ consists of an $H$-module structure $\rho \colon H \otimes A \to A$ on $A$ such that
\begin{equation}
\label{eq:mha}
\rho \circ (\alpha_H^2 \otimes \mu_A) = \mu_A \circ \rho^2.
\end{equation}
We call \eqref{eq:mha} the \emph{module Hom-algebra axiom}.  Here $\rho^2 \colon H \otimes A^{\otimes 2} \to A^{\otimes 2}$ is the map \eqref{eq:rho2'} in Lemma ~\ref{lem:rho2}.

If we write $\rho(x \otimes a) = xa$ for $x \in H$ and $a \in A$, then \eqref{eq:mha} can be written as
\begin{equation}
\label{eq:mha'}
\alpha_H^2(x)(ab) = \sum_{(x)}\, (x'a)(x''b) 
\end{equation}
for $x \in H$ and $a, b \in A$.  If $\alpha_H^2 = Id_H$ (e.g., if $\alpha_H = Id_H)$, then \eqref{eq:mha'} reduces to the usual module algebra axiom
\begin{equation}
\label{eq:maaxiom}
x(ab) = \sum_{(x)}\, (x'a)(x''b).
\end{equation}
In particular, module algebras are examples of module Hom-algebras in which $\alpha = Id$ for both the bialgebra and the algebra involved.

We are now ready to finish the proof of Theorem ~\ref{thm:char}.

\begin{proof}[Proof of Theorem ~\ref{thm:char}]
The first two parts of the Theorem were proved in Lemma ~\ref{lem:rhotilde} and Lemma ~\ref{lem:rho2}.  By multiplicativity, $\mu_A \colon A^{\otimes 2} \to A$ is a morphism of Hom-modules.  Now equip $A$ and $A^{\otimes 2}$ with the $H$-module structures $\rhotilde$ \eqref{eq:rhot} and $\rho^2$ \eqref{eq:rho2'}, respectively.  Then $\mu_A$ is a morphism of $H$-modules if and only if
\[
\begin{split}
\mu_A \circ \rho^2 
&= \rhotilde \circ (Id_H \otimes \mu_A) \quad \text{by \eqref{eq:modmorphism}}\\
&= \rho \circ (\alpha_H^2 \otimes Id_A) \circ (Id_H \otimes \mu_A) \quad \text{by \eqref{eq:rhot}}\\
&= \rho \circ (\alpha_H^2 \otimes \mu_A).
\end{split}
\]
The above equality is exactly the module Hom-algebra axiom \eqref{eq:mha}, as desired.
\end{proof}

\section{Deforming module algebras into module Hom-algebras}
\label{sec:deform}

The purpose of this section is to prove Theorem ~\ref{thm:deform}, and hence also Corollary ~\ref{cor:deform}.  We will also provide some examples to illustrate Corollary ~\ref{cor:deform}.


\begin{proof}[Proof of Theorem ~\ref{thm:deform}]
The Hom-associative algebra $A_\alpha = (A,\mu_{\alpha,A} = \alpha_A \circ \mu_A,\alpha_A)$ was discussed in \S \ref{subsec:homass}.  As for the Hom-bialgebra $H_\alpha$, dualizing the argument for $A_\alpha$, one can check that $\Delta_{\alpha,H} = \Delta_H \circ \alpha_H$ is Hom-coassociative \eqref{eq:homcoass}. The conditions ~\eqref{eq:Deltaalpha} and ~\eqref{eq:Deltamu} follow from the assumptions that $\alpha_H$ is a bialgebra morphism and that $H$ is a bialgebra, respectively.

To show that $\rho_\alpha = \alpha_A \circ \rho$ gives the Hom-associative algebra $A_\alpha$ the structure of an $H_\alpha$-module Hom-algebra, we need to check that (i) $\rho_\alpha$ gives $(A,\alpha_A)$ the structure of an $H$-module and that (ii) the module Hom-algebra axiom \eqref{eq:mha} holds.  The proof that $\rho_\alpha$ is the structure map of an $H$-module structure on $A$ is similar to the proofs of Lemmas ~\ref{lem:rhotilde} and ~\ref{lem:rho2}, so we will leave it to the reader as an easy exercise.

It remains to prove the module Hom-algebra axiom \eqref{eq:mha} in this case, which  states that
\begin{equation}
\label{eq:mhadeform}
\rho_\alpha \circ (\alpha_H^2 \otimes \mu_{\alpha,A}) = \mu_{\alpha,A} \circ \rho_\alpha^2.
\end{equation}
Let us first decipher the map $\rho_\alpha^2$.  From \eqref{eq:rho2'} and $\rho_\alpha = \alpha_A \circ \rho$, we have
\begin{equation}
\label{eq:rhoalpha2}
\begin{split}
\rho_\alpha^2 
&= \rho_\alpha^{\otimes 2} \circ (Id_H \otimes \tau_{H,A} \otimes Id_A) \circ (\Delta_{\alpha,H} \otimes Id_A^{\otimes 2})\\
&= \alpha_A^{\otimes 2} \circ \rho^{\otimes 2} \circ (Id_H \otimes \tau_{H,A} \otimes Id_A) \circ ((\Delta_H \circ \alpha_H) \otimes Id_A^{\otimes 2})\\
&= \rho^{\otimes 2} \circ (Id_H \otimes \tau_{H,A} \otimes Id_A) \circ ((\Delta_H \circ \alpha_H^2) \otimes \alpha_A^{\otimes 2}) \quad \text{by \eqref{eq:alpharho} and \eqref{eq:Deltaalpha}}.
\end{split}
\end{equation}
Since $A$ is an $H$-module algebra, it follows from \eqref{eq:maaxiom} that
\begin{equation}
\label{eq:murhoalpha2}
\mu_A \circ \rho^{\otimes 2} \circ (Id_H \otimes \tau_{H,A} \otimes Id_A) \circ ((\Delta_H \circ \alpha_H^2) \otimes \alpha_A^{\otimes 2}) = \rho \circ (\alpha_H^2 \otimes (\mu_A \circ \alpha_A^{\otimes 2})).
\end{equation}
Therefore, we have
\[
\begin{split}
\rho_\alpha \circ (\alpha_H^2 \otimes \mu_{\alpha,A}) 
&= \alpha_A \circ \rho \circ (\alpha_H^2 \otimes (\mu_A \circ \alpha_A^{\otimes 2})) \quad \text{by multiplicativity of $\alpha_A$}\\
&= \alpha_A \circ \mu_A \circ \rho^{\otimes 2} \circ (Id_H \otimes \tau_{H,A} \otimes Id_A) \circ ((\Delta_H \circ \alpha_H^2) \otimes \alpha_A^{\otimes 2}) \quad \text{by \eqref{eq:murhoalpha2}}\\
&= \mu_{\alpha,A} \circ \rho_\alpha^2 \quad \text{by \eqref{eq:rhoalpha2}}.
\end{split}
\]
This proves \eqref{eq:mhadeform}, as desired.
\end{proof}

We now give some examples that illustrate Corollary ~\ref{cor:deform}, which is the special case of Theorem ~\ref{thm:deform} when $\alpha_H = Id_H$.

\begin{example}
\label{ex:aut}
Let $(A,\mu)$ be an associative algebra.  Denote by $G$ the group of algebra automorphisms of $A$ and by $\bk[G]$ its group bialgebra, in which $\Delta(\varphi) = \varphi \otimes \varphi$ for $\varphi \in G$.  It is easy to check that there is a $\bk[G]$-module algebra structure on $A$ whose structure map $\rho \colon \bk[G] \otimes A \to A$ is given by $\rho(\varphi \otimes a) = \varphi(a)$.

Suppose that $\alpha \colon A \to A$ is an algebra endomorphism such that $\alpha \circ \varphi = \varphi \circ \alpha$ for all $\varphi \in G$.  For example, if $A$ is unital and $a \in A$ is invertible such that $\varphi(a) = a$ for all $\varphi \in G$, then $i_a \circ \varphi = \varphi \circ i_a$, where $i_a(b) = aba^{-1}$.  Such a map $\alpha$ is clearly $\bk[G]$-linear.  Therefore, by Corollary ~\ref{cor:deform}, there is a $\bk[G]$-module Hom-algebra structure
\[
\rho_\alpha = \alpha \circ \rho \colon \bk[G] \otimes A \to A, \quad \rho_\alpha(\varphi \otimes a) = \alpha(\varphi(a))
\]
on the Hom-associative algebra $A_\alpha = (A, \mu_{\alpha,A} = \alpha \circ \mu, \alpha)$.\qed
\end{example}

\begin{example}
\label{ex:steenrod}
Fix a prime $p$, and let $X$ be a topological space.  The singular mod $p$ cohomology $A = \mathrm{H}^*(X; \bZ/p)$ of $X$ is an $\cata_p$-module algebra, where $\cata_p$ is the Steenrod algebra associated to the prime $p$ \cite{es,hatcher}.  Now let $f \colon X \to X$ be a continuous self-map of $X$.  Then the induced map $f^* \colon A \to A$ on mod $p$ cohomology is a map of $\bZ/p$-algebras that respects the Steenrod operations, i.e., $f^*$ is $\cata_p$-linear.  Therefore, by Corollary  ~\ref{cor:deform}, there is an $\cata_p$-module Hom-algebra structure
\[
\rho_f \colon \cata_p \otimes A \to A, \quad \rho_f(Sq^k \otimes x) = f^*(Sq^k(x))
\]
on $A$, where $x \in A$ and $Sq^k \in \cata_p$ is the $k$th Steenrod operation of degree $2k(p-1)$ (resp. $k$) if $p$ is odd (resp. if $p = 2$).

Likewise, we can consider the complex cobordism $\mathrm{MU}^*(X)$ of $X$, which is an $S$-module algebra, where $S$ is the Landweber-Novikov algebra \cite{landweber,novikov} of stable cobordism operations.  The same considerations as in the previous paragraph, now with $A = \mathrm{MU}^*(X)$ and $S$ instead of $\cata_p$, can be applied here.\qed
\end{example}

\section{Twisted $sl(2)$-action on the affine plane}
\label{sec:twisted}

The purposes of this section are to prove Theorem ~\ref{thm:U} and to use this Theorem to construct some $q$-deformations of the $sl(2)$-action on the affine plane.

\begin{proof}[Proof of Theorem ~\ref{thm:U}]
First note that we only need to prove the first part, since the second part follows from it and Theorem ~\ref{thm:deform}.

To prove the first part, observe that, if $\alpha_L$ can be extended to an algebra endomorphism $\alpha_U$ of the universal enveloping algebra $U(L)$, then $\alpha_U$ must be unique because $L$ generates $U(L)$ as an algebra.  Thus, it remains to show that $\alpha_L$ can be extended to a bialgebra endomorphism $\alpha_U$ of $U(L)$ that satisfies \eqref{eq:alphaza}.

We define $\alpha_U \colon U(L) \to U(L)$ by extending $\alpha_L$ linearly and multiplicatively with $\alpha_U(1) = 1$.  To see that this $\alpha_U$ is a well-defined algebra endomorphism, let $x$ and $y$ be elements in $L$.  Then we have
\[
\begin{split}
\alpha_U([x,y] - xy + yx) 
&= \alpha_L([x,y]) - \alpha_L(x)\alpha_L(y) + \alpha_L(y)\alpha_L(x)\\
&= [\alpha_L(x),\alpha_L(y)] - \alpha_L(x)\alpha_L(y) + \alpha_L(y)\alpha_L(x) = 0,
\end{split}
\]
showing that $\alpha_U$ is an algebra endomorphism of $U(L)$.  To check that $\alpha_U$ is a bialgebra endomorphism, we must show that $\Delta_U \circ \alpha_U = \alpha_U^{\otimes 2} \circ \Delta_U$.  Since both $\alpha_U$ and $\Delta_U$ are algebra endomorphisms of $U(L)$, it suffices to check this on $xy$ with $x,y \in L$.  We compute as follows:
\[
\begin{split}
\Delta_U(\alpha_U(xy)) 
&= \Delta_U(\alpha_L(x))\Delta_U(\alpha_L(y))\\
&= (\alpha_L(x) \otimes 1 + 1 \otimes \alpha_L(x))(\alpha_L(y) \otimes 1 + 1 \otimes \alpha_L(y))\\
&= \alpha_U^{\otimes 2}(xy \otimes 1 + 1 \otimes xy + x \otimes y + y \otimes x)\\
&= \alpha_U^{\otimes 2}(\Delta_U(xy)).
\end{split}
\]
This proves that $\alpha_U$ is a bialgebra endomorphism of $U(L)$.

Finally, to check \eqref{eq:alphaza}, it suffices to check it when $z = xy$ for $x, y \in L$.  We have
\[
\begin{split}
\alpha_A((xy)a) 
&= \alpha_A(x(ya)) = \alpha_L(x)\alpha_A(ya) = \alpha_L(x)(\alpha_L(y)\alpha_A(a))\\
&= (\alpha_L(x)\alpha_L(y))\alpha_A(a) = \alpha_U(xy)\alpha_A(a),
\end{split}
\]
as desired.
\end{proof}


In the following example, we use Theorem ~\ref{thm:U} to construct some $q$-deformations of the $sl(2)$-action on the affine plane.

\begin{example}
\label{ex:q}
Let us first recall the usual, and the most important, $sl(2)$-action on the affine plane $A = \bk[x,y]$, where $\bk$ is assumed to have characteristic $0$.  The reader may consult, e.g., \cite[Chapter V]{kassel}, for the details.  Denote by $sl(2)$ the Lie algebra (under the commutator bracket) of $2 \times 2$-matrices with entries in $\bk$ and trace $0$.  It has a standard basis consisting of the matrices
\[
X = \begin{pmatrix}0 & 1\\ 0 & 0\end{pmatrix},\quad
Y = \begin{pmatrix}0 & 0\\ 1 & 0\end{pmatrix},\quad
Z = \begin{pmatrix}1 & 0\\ 0 & -1\end{pmatrix},
\]
satisfying 
\[
[X,Y] = Z, \quad [X,Z] = -2X, \quad [Y,Z] = 2Y.  
\]
Let $H$ denote its universal enveloping bialgebra $U(sl(2))$.  Then there is an $H$-module algebra structure on $A = \bk[x,y]$ determined by
\begin{equation}
\label{eq:XP}
XP = x\frac{\partial P}{\partial y},\quad
YP = y\frac{\partial P}{\partial x},\quad
ZP = x\frac{\partial P}{\partial x} - y\frac{\partial P}{\partial y},
\end{equation}
where $P \in A$ and $\partial/\partial x$ and $\partial/\partial y$ are the formal partial derivatives.  What is special about this $H$-module algebra is that it captures all the finite dimensional simple $sl(2)$-modules.  Indeed, for $n \geq 0$, let $A_n$ denote the $(n+1)$-dimensional subspace of $A$ consisting of the homogeneous polynomials of degree $n$.  Then $A_n$ is a sub-$H$-module of $A$ that is isomorphic to the unique (up to isomorphism) simple $sl(2)$-module $V(n)$ of dimension $n + 1$.

We now deform the above $H$-module algebra structure on $A$ into a module Hom-algebra using Theorem ~\ref{thm:U}.  Fix a non-zero scalar $q \in \bk$.  Define an algebra endomorphism $\alpha_A \colon A \to A$ on the affine plane $A$ by setting
\[
\alpha_A(x) = q^2x \quad\text{and}\quad \alpha_A(y) = qy.
\]
Also define a linear map $\alpha_L \colon sl(2) \to sl(2)$ by setting
\[
\alpha_L(X) = qX,\quad
\alpha_L(Y) = q^{-1}Y,\quad
\alpha_L(Z) = Z.
\]
It is easy to check that $\alpha_L$ is a Lie algebra endomorphism on $sl(2)$.  To use Theorem ~\ref{thm:U}, it suffices to check that
\begin{equation}
\label{eq:alphaWP}
\alpha_A(WP) = \alpha_L(W)\alpha_A(P)
\end{equation}
for $P \in A = \bk[x,y]$ and $W \in \{X,Y,Z\} \subseteq sl(2)$.  Note that 
\[
\alpha_A(P) = P(q^2x,qy).
\]
For example, if $P$ is the monomial $x^iy^j$, then 
\[
\alpha_A(P) = (q^2x)^i(qy)^j = q^{2i+j}x^iy^j = q^{2i+j}P.
\]
We check \eqref{eq:alphaWP} for $W = X$; the other two cases ($W = Y$ and $W = Z$) are essentially identical.  Using \eqref{eq:XP}, we compute as follows:
\[
\begin{split}
\alpha_A(XP) 
&= q^2x \cdot \left(\frac{\partial P}{\partial y}\right)(q^2x,qy) 
= qx \cdot \left(\frac{\partial P}{\partial y}\right)(q^2x,qy) \cdot \frac{d(qy)}{dy}\\
&= qx \cdot \frac{\partial (P(q^2x,qy))}{\partial y} = qX(P(q^2x,qy)) 
= \alpha_L(X)\alpha_A(P).
\end{split}
\]

Thus, Theorem ~\ref{thm:U} applies here with $L = sl(2)$, $A = \bk[x,y]$, and $\rho$ being determined by \eqref{eq:XP}.  In other words, with the notations in Theorem ~\ref{thm:U}, there is an $H_\alpha$-module Hom-algebra structure $\rho_\alpha = \alpha_A \circ \rho \colon H \otimes A \to A$ on the Hom-associative algebra $A_\alpha$.  The structure map $\rho_\alpha$ is determined by
\[
\begin{split}
\rho_\alpha(X \otimes P) &= q^2x \cdot \left(\frac{\partial P}{\partial y}\right)(q^2x,qy), \quad
\rho_\alpha(Y \otimes P) = qy \cdot \left(\frac{\partial P}{\partial x}\right)(q^2x,qy),\\
\rho_\alpha(Z \otimes P) &= q^2x \cdot \left(\frac{\partial P}{\partial x}\right)(q^2x,qy) - qy \cdot \left(\frac{\partial P}{\partial y}\right)(q^2x,qy).
\end{split}
\]
Of course, if $q = 1 \in \bk$, then $\alpha_A = Id_A$, $\alpha_L = Id_L$, and $\rho_\alpha = \rho$ is the original structure map \eqref{eq:XP} of the $H$-module algebra $A$.
\qed
\end{example}




\begin{thebibliography}{AA}
\bibitem{abe}
E. Abe, Hopf algebras, Cambridge Tracts in Math. \textbf{74}, Cambridge U. Press, Cambridge, 1977.

\bibitem{atm}
H. Ataguema, A. Makhlouf, and S. Silvestrov, Generalization of $n$-ary Nambu algebras and beyond, \texttt{arXiv:0812.4058v1}.

\bibitem{es}
D.B.A. Epstein and N.E. Steenrod, Cohomology operations, Ann. Math. Studies \textbf{50}, Princeton U. Press, Princeton, 1962.

\bibitem{fmy}
Y. Fr\'egier, M. Markl, and D. Yau, The $L_\infty$-deformation complex of diagrams of algebras, \texttt{arXiv:0812.2981}.

\bibitem{hls}
J.T. Hartwig, D. Larsson, and S.D. Silvestrov, Deformations of Lie algebras using $\sigma$-derivations, J. Algebra \textbf{295} (2006), 314-361.

\bibitem{hatcher}
A. Hatcher, Algebraic topology, Cambridge U. Press, Cambridge, 2002.

\bibitem{hu}
N. Hu, $q$-Witt algebras, $q$-Lie algebras, $q$-holomorph structure and representations, Alg. Colloq. \textbf{6} (1999), 51-70.

\bibitem{kassel}
C. Kassel, Quantum groups, Grad. Texts in Math. \textbf{155}, Springer-Verlag, New York, 1995.

\bibitem{landweber}
P.S. Landweber, Cobordism operations and Hopf algebras, Trans. Amer. Math. Soc. \textbf{129} (1967), 94-110.

\bibitem{liu}
K. Liu, Characterizations of quantum Witt algebra, Lett. Math. Phy. \textbf{24} (1992), 257-265.

\bibitem{ms}
A. Makhlouf and S. Silvestrov, Hom-algebra structures, J. Gen. Lie Theory Appl. \textbf{2} (2008), 51-64.

\bibitem{ms2}
A. Makhlouf and S. Silvestrov, Hom-Lie admissible Hom-coalgebras and Hom-Hopf algebras, S. Silvestrov et. al. eds., Gen. Lie theory in Math., Physics and Beyond, Ch. 17, pp. 189-206, Springer-Verlag, Berlin, 2008.

\bibitem{ms3}
A. Makhlouf and S. Silvestrov, Notes on formal deformations of Hom-associative and Hom-Lie algebras, to appear in Forum Math.  Preprints in Math. Sci., Lund Univ., Center for Math. Sci., 2007. \texttt{arXiv:0712.3130v1}.

\bibitem{ms4}
A. Makhlouf and S. Silvestrov, Hom-algebras and Hom-coalgebras, Preprints in Math. Sci., Lund Univ., Center for Math. Sci., 2008. \texttt{arXiv:0811.0400v2}.

\bibitem{markl}
M. Markl, Operads and PROPS, Handbook of Algebra {\bf 5}, 87-140, Elsevier, 2008.  \texttt{arXiv:math.AT/0601129v3}.

\bibitem{mont}
S. Montgomery, Hopf algebras and their actions on rings, CBMS Regional Conference Series in Math. \textbf{82}, Amer. Math. Soc., Providence, 1993.

\bibitem{novikov}
S.P. Novikov, Methods of algebraic topology from the point of view of cobordism theory, Izv. Akad. Nauk SSSR Ser. Mat. \textbf{31} (1967), 855--951.

\bibitem{sweedler}
M. Sweedler, Hopf algebras, W.A. Benjamin, New York, 1969.

\bibitem{yau}
D. Yau, Enveloping algebras of Hom-Lie algebras, J. Gen. Lie Theory Appl. \textbf{2} (2008), 95-108.

\bibitem{yau2}
D. Yau, Hom-algebras and homology, \texttt{arXiv:0712.3515v2}.

\bibitem{yau3}
D. Yau, Hom-bialgebras and comodule algebras, \texttt{arXiv:0810.4866}.

\end{thebibliography}
\end{document}